\documentclass[12pt]{amsart}

\usepackage{mathrsfs}
\usepackage{graphicx}
\usepackage{amsmath}
\usepackage{amssymb}
\usepackage{amsthm}
\usepackage{graphicx}

\newtheorem{theorem}{Theorem}
\theoremstyle{plain}

\newtheorem{proposition}[theorem]{Proposition}
\newtheorem{corollary}[theorem]{Corollary}

\newcommand{\bC}{\mathbb{C}}

\newcommand{\bE}{\mathbb{E}}

\def\dd{{\mathrm{d}}}

\newcommand\bel[1]{\begin{equation}\label{#1}}
\newcommand\ee{\end{equation}}
\newcommand\bld[1]{\boldsymbol{#1}}

\numberwithin{equation}{section}
\numberwithin{theorem}{section}

\begin{document}

\today

\title[Kronecker products]
{Simple  Spectral Bounds
for Sums of  Certain Kronecker Products}

\author{S. V. Lototsky}
\curraddr[S. V. Lototsky]
{Department of Mathematics, USC\\
Los Angeles, CA 90089 USA\\
tel. (+1) 213 740 2389; fax: (+1) 213 740 2424}
\email[S. V. Lototsky]{lototsky@usc.edu}
\urladdr{http://www-bcf.usc.edu/$\sim$lototsky}


\begin{abstract}
New bounds are derived for the
  eigenvlues of sums of  Kronecker products of square matrices
  by relating the corresponding matrix expressions to the
covariance structure of suitable bi-linear stochastic systems in discrete and continuous
time.
  \end{abstract}

\maketitle

\section{Introduction}

 Kronecker product  reduces a matrix-matrix equation to an equivalent
 matrix-vector form (\cite{BJ} or  \cite[Chapter 4]{HJ}).
For example,  consider a matrix equation $BXA^{\top}=C$
with known $\dd$-by-$\dd$ matrices $A,B,C$, and
the unknown $\dd$-by-$\dd$  matrix $X$. To cover the most general setting, all matrices
are assumed to have complex-valued entries.
  Introduce a column vector
  $\mathrm{vec}(X)=\bld{X}\in \bC^{\dd^2}$ by stacking together the columns of
  $X$, left-to-right:
  \begin{equation}
  \label{vec}
 \mathrm{vec}(X)=
 \bld{X}=(X_{11},\ldots,X_{\dd 1},X_{12},\ldots,X_{\dd 2},\ldots,
  X_{1 \dd},\ldots, X_{\dd \dd})^{\top}.
  \end{equation}
  Then direct computations show that
  the matrix equation $AXB^{\top} =C$ can be written in the matrix-vector
  form for the unknown vector $\bld{X}$ as
  \begin{equation}
  \label{KP0}
  (A\otimes B) \bld{X}=\bld{C},\ \bld{C}=\mathrm{vec}(C),
  \end{equation}
where $A\otimes B$ is the {\tt Kronecker product} of
matrices $A$ and $B$, that is, an $\dd^2$-by-$\dd^2$ block matrix with
blocks $A_{ij}B$. In other words, \eqref{KP0} means
\begin{equation}
  \label{KP}
  \mathrm{vec}\big(BXA^{\top})=(A\otimes B)\mathrm{vec}(X),
\end{equation}
with $\mathrm{vec}(\cdot)$ operation defined in \eqref{vec}.

In what follows, an
 $\dd$-dimensional column vector will usually be denoted by a lower-case
 bold Latin letter, e.g. $\bld{h}$, whereas upper-case regular
 Latin letter, e.g. $A$, will usually mean an $\dd$-by-$\dd$ matrix.
  Then $|\bld{h}|$ is the Euclidean norm of $\bld{h}$ and $|A|$
  is the induced matrix norm
  $$
  |A|=\max\big\{|A\bld{h}|:|\bld{h}|=1\big\}.
   $$
   For a matrix $A\in \bC^{\dd\times \dd}$,
  $\overline{A}$ is the matrix with complex conjugate entries,
  $A^{\top}$ means  transposition, and $A^*$ denotes the
  conjugate transpose: $A^*=\overline{A^{\top}}={\overline{A}}^{\top}$.
  The same notations, $\overline{\phantom{A}}$, ${}^\top$, and ${}^*$,
  will also be used for column vectors in $\bC^\dd$. The identity matrix
  is $I$.

For a square matrix $A$, define the following numbers:
\begin{align*}
\bld{\rho}(A)&
=\max\{|\lambda(A)|: \lambda(A) {\rm \ is \ an \ eigenvalue\
of } \ A\} \ \ ({\rm {\tt spectral\  radius}\ of}\ A);
\\
\bld{\alpha}(A) &
=\max\{ \Re \lambda(A): \lambda(A) {\rm \ is \ an \ eigenvalue\
of } \ A\} \ \ ({\rm {\tt spectral \ abscissa}\ of}\ A);
\\
\bld{\varrho}(A)&
=\min\{ \Re \lambda(A): \lambda(A) {\rm \ is \ an \ eigenvalue\
of } \ A\}.
\end{align*}
For a Hermitian matrix $H$,
\begin{equation}
\label{bound1}
\bld{\varrho}(H)|\bld{x}|^2\leq \bld{x}^*H\bld{x}\leq \bld{\alpha}(H)|\bld{x}|^2.
\end{equation}

While eigenvalues of the matrices
$
A\otimes B \ {\rm and }\  A\otimes I + I\otimes B
$
can be  easily expressed in terms of the eigenvalues of the matrices
$A$ and $B$ \cite[Theorems 4.2.12 and 4.4.5]{HJ},
there is, in general, no easy way to get the eigenvalues of the matrices
\begin{equation}
\label{DTM}
D_{A,B}=\overline{A}\otimes  A + \sum_{k=1}^m\overline{B}_k\otimes B_k
\end{equation}
and
\begin{equation}
\label{CTM}
C_{A,B}=\overline{A}\otimes I + I\otimes A +
\sum_{k=1}^m\overline{B}_k\otimes B_k,
\end{equation}
 which appear, for example, in the study of bi-linear stochastic systems.
Paper \cite{FLW} presents one of the first investigations of the
spectral properties of \eqref{DTM} and \eqref{CTM}.
 The main result of the current paper provides another contribution
  to the subject:
\begin{theorem}
\label{th-main}
Given matrices $A,B_1, \ldots, B_m \in \mathbb{C}^{\dd\times \dd}$,
define the matrix $D_{A,B}$ by \eqref{DTM}, the matrix
$C_{A,B}$ by \eqref{CTM}, and also the matrices
\begin{equation}
\label{DTN}
N_{A,B}=A^*A+\sum_{k=1}^m B^*_kB_k
\end{equation}
and
\begin{equation}
\label{CTMM}
M_{A,B}= A+A^*+\sum_{k=1}^m B^*_kB_k.
\end{equation}
Then
\begin{align}
\label{DTSKS}
\bld{\varrho}(N_{A,B})&\leq \bld{\rho}(D_{A,B})
 \leq \bld{\alpha}(N_{A,B}),\\
 \label{CTSKS}
\bld{\varrho}(M_{A,B})&\leq \bld{\alpha}(C_{A,B})
 \leq \bld{\alpha}(M_{A,B}).
\end{align}
\end{theorem}
In the particular case of real matrices and $m=1$, Theorem \ref{th-main} implies
\begin{align*}
&\bld{\varrho}\left(A^{\top}A+ B^{\top}B\right)\leq
\bld{\rho}\left({A}\otimes A+
{B}\otimes B\right)
 \leq \bld{\alpha}\left(A^{\top}A+ B^{\top}B\right),\\
&\bld{\varrho}\left(A+A^{\top}+ B^{\top}B\right)\leq
\bld{\alpha}\left({A}\otimes I + I\otimes A
+ {B}\otimes B\right)
\leq  \bld{\alpha}\left(A+A^{\top}+ B^{\top}B\right).
\end{align*}

\begin{corollary} If the matrix $N_{A,B}$ is scalar, that is,
$N_{A,B}=\beta I$, then $\bld{\rho}(D_{A,B})=\beta$;
if $M_{A,B}=\beta I$, then $\bld{\alpha}(C_{A,B})=\beta$.
\end{corollary}

The reason Theorem \ref{th-main} is potentially useful is that
the matrices $M_{A,B}$ and $N_{A,B}$ are
Hermitian and have size $\dd$-by-$\dd$,
whereas the matrices $C_{A,B}$ and $D_{A,B}$ are in general
not Hermitian or even normal and have a much bigger size $\dd^2$-by-$\dd^2$.
For example, with $m=1$, if matrices $A$ and $B$ are orthogonal, then
the matrix $D_{A,B}$ can be fairly complicated, but
$N_{A,B}=2I$, and  we immediately conclude that $\bld{\rho}(D_{A,B})=2$.
Similarly, let  $m=1$, let
 $A=a I+S$ for a real number $a$ and a skew-symmetric
matrix $S$, and let $B$ be orthogonal, then $\bld{\alpha}(C_{A,B})=2a+1$.
Section \ref{secEx} below presents more examples and further
discussions.

The matrix expressions $A\otimes B$ and $A\otimes I + I\otimes B$
have designated names (Kronecker product and Kronecker sum),
but there is no established terminology for  \eqref{DTM} and \eqref{CTM}.
In what follows, \eqref{DTM} will be referred to as the discrete-time
stochastic Kronecker sum, and \eqref{CTM} will be referred to
as the continuous-time stochastic Kronecker sum. The reason for this
choice of names is motivated by the type of problems
in which the corresponding matrix expressions appear.

The proof of Theorem \ref{th-main} relies on the analysis of the covariance
matrix of suitably constructed random vectors. Recall that
the {\tt covariance matrix} of
 two $\bC^\dd$-valued random column-vectors
 $\bld{x}=(x_1,\ldots,x_\dd)^{\top}$ and $\bld{y}=(y_1,\ldots,y_\dd)^{\top}$
is
$
U_{x,y}=\bE \bld{x}\bld{y}^*.
$
Also define
$
r_x=\sum_{i=1}^\dd \bE |x_i|^2,\
\ r_y=\sum_{i=1}^\dd \bE |y_i|^2,\ {\rm and } \
\bld{U}_{x,y}={\mathrm{vec}}(U_{x,y}).
$
Then
$
|\bld{U}_{x,y}|^2 = \sum_{i,j=1}^\dd|\bE x_iy^*_j|^2,
$
 and the Cauchy-Schwartz inequality
$
|\bE x_iy^*_j|^2\leq
 \bE|x_i|^2 \bE|y_j|^2
$
leads to an upper bound on $|\bld{U}_{x,y}|$:
\begin{equation}
\label{cov}
|\bld{U}_{x,y}|^2 \leq r_xr_y.
\end{equation}
In the special case $\bld{x}=\bld{y}$,
\begin{equation*}
\begin{split}
\dd|\bld{U}_{x,y}|^2 &= \dd\sum_{i,j=1}^\dd|\bE x_ix^*_j|^2
= \dd\sum_{i=1}^\dd \Big(\bE |x_i|^2 \Big)^2 +
\dd\sum_{i\not= j }|\bE x_ix^*_j|^2\\
&
\geq \dd\sum_{i=1}^\dd \Big(\bE |x_i|^2 \Big)^2 \geq
\left(\sum_{i=1}^\dd \bE|x_i|^2\right)^2,
\end{split}
\end{equation*}
leading to  a lower bound:
\begin{equation}
\label{cov1}
|\bld{U}_{x,x}|\geq {\dd}^{-1/2} \,r_x.
\end{equation}

Section \ref{secDT} explains how matrices of the type \eqref{DTM}
appear in the analysis of discrete-time bi-linear stochastic systems and
presents the proof of \eqref{DTSKS}.
Section \ref{secCT} explains how matrices of the type \eqref{CTM}
appear in the analysis of continuous-time bi-linear stochastic systems and
presents the proof of \eqref{CTSKS}. The connection with stochastic
systems also illustrates why it is indeed natural to bound
the spectral radius for matrices of the type \eqref{DTM} and the
spectral abscissa for matrices of the type \eqref{CTM}.

\section{Discrete-Time Stochastic Kronecker Sum}
\label{secDT}
Given matrices $A,B_1, \ldots, B_m \in \mathbb{C}^{\dd\times \dd}$,
consider two $\mathbb{C}^\dd$-valued random sequences
$\bld{x}(n)=(x_1(n),\ldots, x_\dd(n))^{\top}$ and
$\bld{y}(n)=(y_1(n),\ldots, y_\dd(n))^{\top}$,\ $n=0,1,2,\ldots,$ defined by
\begin{equation}
\label{DT1}
\begin{split}
\bld{x}(n+1)&=A\bld{x}(n)+\sum_{k=1}^m B_k\bld{x}(n)\xi_{n+1,k},
\ \bld{x}(0)=\bld{u},\\
\bld{y}(n+1)&=A\bld{y}(n)+\sum_{k=1}^m B_k\bld{y}(n)\xi_{n+1,k},
\ \bld{y}(0)=\bld{v}.
\end{split}
\end{equation}
Both equations in \eqref{CT1} are driven by a
{\tt white noise} sequence $\xi_{n,k}$, $n\geq 1,\
k=1,\ldots, m$ of independent, for all $n$ and $k$, random variables,
all with zero mean and unit variance:
\begin{equation}
\label{wn-dt}
\bE \xi_{n,k}=0,\ \bE\xi_{n,k}^2=1,\
\bE \xi_{n,k} \xi_{p,\ell}=0 \ {\rm if } \ n\not=p \ {\rm or}\
k\not= \ell;
\end{equation}
 the initial conditions $\bld{u},\bld{v}\in \mathbb{C}^N$ are non-random.
 Note that the sequences $\bld{x}(n)$ and
$\bld{y}(0)$ satisfy the same equation and differ only in the initial
conditions. In particular, $\bld{u}=\bld{v}$ implies
$\bld{x}(n)=\bld{y}(n)$ for all $n\geq 0$.  The term {\em bi-linear} in connection with \eqref{DT1} reflects the
fact that the noise sequence enters the system in a multiplicative, as opposed to
additive, way.

\begin{proposition}
\label{Prop-dtc}
Define  $V(n)=\bE \bld{x}(n)\bld{y}^*(n)$,  the covariance matrix
of the random vectors $\bld{x}(n)$ and $\bld{y}(n)$ from \eqref{DT1},
and define $r_x(n)=\bE \bld{x}^*(n)\bld{x}(n) = \bE|\bld{x}(n)|^2$.
  Then the vector
$
\bld{U} (n)=\mathrm{vec}\big(V (n)\big)
$
satisfies
\begin{equation}
\label{DT3}
{\bld{U}} (n+1)=D_{A,B}^n\bld{U} (0),
\end{equation}
with the matrix
\begin{equation}
\label{DT4}
D_{A,B}=\overline{A}\otimes A
+\sum_{k=1}^m \overline{B}_k\otimes B_k,
\end{equation}
and the number $r_x(n)$ satisfies
\begin{equation}
\label{DT6.1}
|\bld{u}|^2\gamma^n \leq r_x(n)\leq |\bld{u}|^2\beta^n,
\end{equation}
where $\gamma$ is the smallest eigenvalue and $\beta$ is the largest
eigenvalue of the non-negative Hermitian matrix
\begin{equation}
\label{DT-NN}
N_{A,B}=A^*A+\sum_{k=1}^m B_k^*B_k.
\end{equation}

\end{proposition}

\begin{proof}
By \eqref{DT1},
$$
\bld{x}(n+1)=A\bld{x}(n)+\sum_{k=1}^m B_k\bld{x}(n)\xi_{n+1,k},\ \
 \bld{y}^*(n+1)=\bld{y}^*({n})A^*+\sum_{k=1}^m \bld{y}^*(n)B_k^*\xi_{n+1,k},
$$
so that
\begin{align}
\label{D-prod1}
\bld{x}(n+1)\bld{y}^*(n+1)&=
A\bld{x}(n)\bld{y}^*({n})A^*+
\sum_{k,\ell=1}^m B_k\bld{x}(n)\bld{y}^*(n)B_\ell^*\xi_{n+1,k}
\xi_{n+1,\ell}\\
\label{D-prod2}
&+\sum_{k=1}^m A\bld{x}(n)\bld{y}^*(n)B_k^*\xi_{n+1,k}
+\sum_{k=1}^m B_k\bld{x}(n)\bld{y}^*({n})A^*\xi_{n+1,k}.
\end{align}
The vectors  $\bld{x}(n)$ and $\bld{y}(n)$
are independent of every $\xi_{n+1,k}$.
Therefore, using \eqref{wn-dt},
\begin{align}
\label{D-prod3}
\bE \big(A\bld{x}(n)\bld{y}^*(n)B_k^*\xi_{n+1,k}\big) &=
\bE \big(A\bld{x}(n)\bld{y}^*(n)B_k^*\big)\bE\xi_{n+1,k}=0,
\\
\notag
\sum_{k,\ell=1}^m
\bE \big(B_k\bld{x}(n)\bld{y}^*(n)B_\ell^*\xi_{n+1,k}\xi_{n+1,\ell}\big)
&=\sum_{k,\ell=1}^m
\bE \big(B_k\bld{x}(n)\bld{y}^*(n)B_\ell^*\big)
\bE\big(\xi_{n+1,k}\xi_{n+1,\ell}\big)
\\
\label{D-prod4}
&=\sum_{k=1}^m  B_k\bE \big(\bld{x}(n)\bld{y}^*(n)\big)B_\ell^*
=\sum_{k=1}^m  B_kV (n)B_\ell^*.
\end{align}

As a result,
\begin{equation*}
\label{DT2}
{V }(n+1)
=AV (n)A^*+\sum_{k=1}^m B_kV (n)B_k^*,
\end{equation*}
and \eqref{DT3} follows from \eqref{KP}.

Similarly,
$$
{r}_x(n+1) =
\bE \bld{x}^*(n)\left(A^*A+\sum_{k=1}^m B_k^*B_k\right) \bld{x}(n).
$$
Then  \eqref{bound1} implies $\gamma r_x(n)\leq r_x(n+1)\leq \beta r_x(n)$,
and \eqref{DT6.1} follows.
\end{proof}

\noindent Given the origin of equation \eqref{DT3}, the matrix $D_{A,B}$
from \eqref{DT4} is
natural to call the {\tt {discrete-time stochastic Kronecker sum}}
of the matrices $A$ and $B_k$.

For a square matrix $A$, denote by $\bld{\rho}$ the {\tt spectral
radius} of $A$:
$$
\bld{\rho}(A) =\max\{ |\lambda(A)|: \lambda(A) {\rm \ is \ an \ eigenvalue\
of } \ A\}.
$$
It is really very well known  that
\begin{equation}
\label{SPR}
\bld{\rho}(A) = \lim_{n\to + \infty}  |A^n|^{1/n}.
\end{equation}

\begin{theorem} For every matrices
$A,B_1,\ldots, B_m\in \mathbb{C}^{\dd\times \dd}$,
\begin{equation}
\label{DT5}
\bld{\varrho}(N_{A,B})\leq \bld{\rho}(D_{A,B})\leq
\bld{\alpha}(N_{A,B}),
\end{equation}
where the matrix $N_{A,B}=A^*A+\sum_{k=1}^m B_k^*B_k$,
$\bld{\varrho}(N_{A,B})$ is the smallest eigenvalue of $N_{A,B}$,
and $\bld{\alpha}(N_{A,B})$ is the largest eigenvalue of $N_{A,B}$
\end{theorem}

\begin{proof}
Similar to Proposition \ref{Prop-dtc}, write
$\gamma=\bld{\varrho}(N_{A,B})$ and
$\beta=\bld{\alpha}(N_{A,B}).$
It follows from \eqref{DT3} that
\begin{equation}
\label{DT5.1}
|\bld{U} (n)|=|D^n_{A,B}\,\bld{U} (0)|.
\end{equation}
To get the upper bound in \eqref{DT5}, note that  \eqref{cov} and \eqref{DT6.1}
imply
\begin{equation}
\label{DT6}
|\bld{U} (n)|\leq \sqrt{r_x(n)r_y(n)} \leq |\bld{u}|\, |\bld{v}|\ \beta^n.
\end{equation}

Combining  \eqref{DT5.1} and  \eqref{DT6} leads to
\begin{equation}
\label{DT7}
|D^n_{A,B}\,\bld{U} (0)|\leq |\bld{u}|\, |\bld{v}|\ \beta^n.
\end{equation}
Since $\bld{U} (0)=\mathrm{vec}(\bld{u}\bld{v}^*)=\bld{v}\otimes \bld{u}$,
 and $\bld{u}$ and $\bld{v}$ are arbitrary vectors in $\mathbb{C}^\dd$,
 it follows from
\eqref{DT7} that
\begin{equation}
\label{DT8}
|D^n_{A,B}|\leq a \, \beta^n
\end{equation}
for a positive real number $a$. Then the upper bound in \eqref{DT5} follows from
\eqref{DT8} and \eqref{SPR}.

To get the lower bound, take $\bld{u}=\bld{v}$ with $ |\bld{u}|=1$ so that  $\bld{x}(n)=\bld{y}(n)$
 for all $n\geq 0$. Then \eqref{cov1} and \eqref{DT5.1} imply
 $$
 \dd^{-1/2} \gamma^n \leq |\bld{U} (n)|  \leq |D^n_{A,B}|,
 $$
 and the lower  bound in \eqref{DT5} follows from
 \eqref{SPR}.

\end{proof}

\section{Continuous-Time Stochastic Kronecker Sum}
\label{secCT}
Given matrices $A,B_1, \ldots, B_k \in \mathbb{C}^{\dd\times \dd}$,
consider two $\mathbb{C}^\dd$-valued stochastic processes
$\bld{x}(t)=({x}_1(t),\ldots, x_\dd(t))^{\top}$
and $\bld{y}(t)=(y_1(t),\ldots, y_\dd(t))^{\top}$, $t\geq 0$, defined by the
It\^{o} integral equations
\begin{equation}
\label{CT1}
\begin{split}
\bld{x}(t)&=\bld{u}+\int_0^tA\bld{x}(s)ds
+\sum_{k=1}^m \int_0^tB_k\bld{x}(s)dw_k(s),
\\
\bld{y}(t)&=\bld{v}+\int_0^tA\bld{y}(s)ds+\sum_{k=1}^m \int_0^tB_k\bld{y}(s)dw_k(s).
\end{split}
\end{equation}
Both equations in \eqref{CT1} are driven by
independent standard Brownian motions $w_1$$,\ldots,$ $w_m$,  and the initial
conditions $\bld{u},\bld{v}\in \mathbb{C}^\dd$ are non-random.
 Note that the processes $\bld{x}(t)$ and
$\bld{y}(t)$ satisfy the same equation and differ only in the initial
conditions. Existence and uniqueness of solution are well-known:
\cite[Theorem 5.2.1]{Oksendal}. The terms $dw_k(t)$ can be considered
continuous-time analogues of white noise input in \eqref{DT1}.
 The term {\em bi-linear} in connection with \eqref{CT1} reflects the
fact that the noise process enters the system in a multiplicative, as opposed to
additive, way.

The differential form
$$
d\bld{x}(t)=A\bld{x}(t)dt
+\sum_{k=1}^m B_k\bld{x}(t)dw_k(t),\
d\bld{y}(t)=A\bld{y}(t)dt
+\sum_{k=1}^m B_k\bld{y}(t)dw_k(t)
$$
 is a more compact, and less formal, way to write \eqref{CT1}.

The peculiar behavior of white noise in continuous time, often
written informally as $(dw(t))^2=dt$,  makes it necessary
to modify the usual product rule for the derivatives. The result is know as
the {\tt It\^{o} formula}; its one-dimensional version is
 presented below for the convenience
of the reader.

\begin{proposition}
If $a,b,\sigma$, and $\mu$ are globally Lipschits  continuous functions
and $f(0)$, $g(0)$ are non-random,
then

(a)  there are unique continuous random processes $f$ and $g$ such that
\begin{align*}
f(t)=f(0)+\int_0^t a(f(s))ds+\int_0^t \sigma(f(s))dw(s),\\
g(t)=g(0)+\int_0^t b(g(s))ds+\int_0^t \mu(g(s))dw(s);
\end{align*}

(b) the following equality holds:
\begin{equation}
\label{Ito1}
\bE f(t)g(t)= f(0)g(0)
+\int_0^t \bE\big(f(s)b(g(s))+g(s)a(f(s))+\sigma(f(s))\mu(g(s))\big)ds.
\end{equation}
\end{proposition}

\begin{proof}
In differential form,
$$
d(fg)=fdg+gdf+\sigma\mu\,dt,
$$
where the first two terms on the right come from the usual product rule and the
third term, known as the It\^{o} correction, is a consequence of
$(dw(t))^2=dt$. The expected value of stochastic integrals is zero:
$$
\bE\int_0^t f(s)\mu(g(s))dw(s)=\bE\int_0^t g(s)\sigma(f(s))dw(s)=0,
$$
and then \eqref{Ito1} follows.
 For  more details, see, for example,
\cite[Chapter 4]{Oksendal}.
\end{proof}

\begin{proposition}
\label{prop-CT}
Define $V (t)=\bE \bld{x}(t)\bld{y}^*(t)$,  the covariance matrix
of the random vectors $\bld{x}(t)$ and $\bld{y}(t)$ from \eqref{CT1},
and define $r_x(t)=\bE \bld{x}^*(t)\bld{x}(t)$.  Then
 the vector
$$
\bld{U} (t)=\mathrm{vec}\big(V (t)\big)
$$
satisfies
\begin{equation}
\label{CT3}
{\bld{U}} (t)=e^{t C_{A,B}}\bld{U} (0),
\end{equation}
with the matrix
\begin{equation}
\label{CT4}
C_{A,B}=\overline{A}\otimes I + I\otimes A
+\sum_{k=1}^m \overline{B}_k\otimes B_k,
\end{equation}
and the number $r_x(t)$ satisfies
\begin{equation}
\label{CT6.1}
|\bld{u}|^2e^{\gamma t} \leq r_x(t)\leq |\bld{u}|^2e^{\beta t},
\end{equation}
where $\gamma$ is the smallest eigenvalue and $\beta$ is the largest
eigenvalue of the  Hermitian matrix
\begin{equation}
\label{CT-MM}
M_{A,B}=A+A^*+\sum_{k=1}^m B_k^*B_k.
\end{equation}
\end{proposition}

\begin{proof}
In differential form,
$$
d\bld{x}(t)=A\bld{x}(t)dt
+\sum_{k=1}^m B_k\bld{x}(t)dw_k(t),\
d\bld{y}^*(t)=\bld{y}^*(t)A^*dt
+\sum_{k=1}^m \bld{y}^*(t)B_k^*dw_k(t).
$$
By the It\^{o} formula,
\begin{equation*}
\label{CT2}
V(t)
=V(0)+
\int_0^t \left(AV (s)+V(s)A^*
+\sum_{k=1}^m B_kV(s)B_k^*\right)ds,
\end{equation*}
and \eqref{CT3} follows from \eqref{KP}.

Similarly,
$$
{r}_x(t) = r_x(0)+\int_0^t
\bE \bld{x}^*(s)M_{A,B} \bld{x}(s)ds,
$$
and then, for every real number $a$,
$$
r_x(t)=r_x(0)+\int_0^t ar_x(s)ds + \int_0^tf_a(s)ds,
$$
where
$$
f_a(s)=\int_0^t \bE \Big( \bld{x}^*(s)M_{A,B} \bld{x}(s)
- a\bld{x}^*(s)\bld{x}(s)\Big)ds.
$$
In other words,
$$
r_x(t)=|\bld{u}|^2 e^{at} +\int_0^t e^{a(t-s)} f_a(s)ds.
$$
If  $a=\gamma$ (the smallest eigenvalue of $M_{A,B}$),
then $f_a(s)\geq 0$ and the lower bound in \eqref{CT6.1}
follows; if $a=\beta$ (the largest eigenvalue of $M_{A,B}$),
then $f_a(s)\leq 0$, and the upper bound in
\eqref{CT6.1} follows.

\end{proof}

\noindent Given the origin of equation \eqref{CT3}, the matrix $C_{A,B}$ is
natural to call the {\tt continuous-} {\tt time stochastic Kronecker sum}
of the matrices $A$ and $B_k$.

For a square matrix $A$, denote by $\bld{\alpha}$ the {\tt spectral
abscissa} of $A$:
$$
\bld{\alpha}(A) =\max\{ \Re \lambda(A): \lambda(A) {\rm \ is \ an \ eigenvalue\
of } \ A\}.
$$
It is known \cite[Theorem 15.3]{Embree} that
\begin{equation}
\label{SPA}
\bld{\alpha}(A) = \lim_{t\to + \infty} \frac{1}{t} \ln |e^{tA}|.
\end{equation}

\begin{theorem}
For every matrices
$A,B_1,\ldots, B_m\in \mathbb{C}^{\dd\times \dd}$,
\begin{equation}
\label{CT5}
\bld{\varrho}(M_{A,B})\leq \bld{\alpha}(C_{A,B})\leq
\bld{\alpha}(M_{A,B}).
\end{equation}
\end{theorem}

\begin{proof}
As in Proposition \ref{prop-CT},
we write
$
\beta=\bld{\alpha}(M_{A,B}),\ \gamma= \bld{\varrho}(M_{A,B}).
$
It follows from \eqref{CT3} that
\begin{equation}
\label{CT5.1}
|\bld{U} (t)|=|e^{tC_{A,B}}\,\bld{U} (0)|.
\end{equation}
By \eqref{cov},
\begin{equation}
\label{CT6}
|\bld{U} (t)|\leq \sqrt{r_x(t)r_y(t)},
\end{equation}
and then \eqref{CT6.1} implies
\begin{equation}
\label{CT7}
|e^{tC_{A,B}}\,\bld{U} (0)|\leq |\bld{u}|\, |\bld{v}|\,e^{\beta t}.
\end{equation}
Since $\bld{U} (0)=\mathrm{vec}(\bld{u}\bld{v}^*)=\bld{v}\otimes \bld{u}$, and $\bld{u}$
and $\bld{v}$ are arbitrary vectors in $\mathbb{C}^\dd$, it follows from
\eqref{CT7} that
\begin{equation}
\label{CT8}
|e^{tC_{A,B}}|\leq b e^{\beta t}
\end{equation}
for a positive real number $b$. Then the upper bound in \eqref{CT5} follows from
\eqref{CT8} and \eqref{SPA}.

To get the lower bound, take $\bld{u}=\bld{v}$ with $|\bld{u}|=1$,
so that  $\bld{x}(t)=\bld{y}(t)$ for all $t\geq 0$.
Then \eqref{cov1} and \eqref{DT6} imply
 $$
 \dd^{-1/2} e^{\gamma t} \leq |\bld{U} (n)|  \leq |e^{tC_{A,B}}|,
 $$
 and the lower  bound in \eqref{CT5} follows from
 \eqref{SPA}.

\end{proof}

\section{Examples and Further Discussions}
\label{secEx}
Without additional information about the matrices $A$ and $B$,
it is not possible to know how tight the bounds in \eqref{DTSKS}
and \eqref{CTSKS} will be. As an example, consider two real matrices
$$
A=
\left(
\begin{array}{ll}
a& 0\\
0& b
\end{array}
\right),\ \
B=
\left(
\begin{array}{ll}
0& 0\\
\sigma & 0
\end{array}
\right).
$$
The corresponding stochastic systems are
$$
x_1(n+1)=ax_1(n),\ x_2(n+1)=bx_2(n)+\sigma x_1(n)\xi_{n+1}
$$
in discrete time, and
$
dx_1(t)=ax_1(t)dt,\ dx_2(t)=bx_2(t)dt+ \sigma x_1(t)dw(t)
$
in continuous time.
Then
\begin{align*}
D_{A,B}&=A\otimes A + B\otimes B=
\left(
\begin{array}{cccc}
a^2 & 0     & 0     & 0 \\
0     & ab &  0    & 0 \\
0     &  0    & ab & 0 \\
\sigma^2 & 0 & 0 & b^2
\end{array}
\right),\\
N_{A,B}&=A^{\top}A + B^{\top}B=
\left(
\begin{array}{cc}
a^2+\sigma^2 & 0\\
0 & b^2
\end{array}
\right);\\
C_{A,B}&=A\otimes I + I\otimes A + B\otimes B=
\left(
\begin{array}{cccc}
2a & 0     & 0     & 0 \\
0     & a+b &  0    & 0 \\
0     &  0    & a+b & 0 \\
\sigma^2 & 0 & 0 & 2b
\end{array}
\right),\\
M_{A,B}&=A^{\top}+A + B^{\top}B=
\left(
\begin{array}{cc}
2a+\sigma^2 & 0\\
0 & 2b
\end{array}
\right).
\end{align*}
In particular, both $\bld{\rho}(D_{A,B})$ and $\bld{\alpha}(C_{A,B})$
do not depend on $\sigma$:
 $$
 \bld{\rho}(D_{A,B})=\max(a^2,b^2),\
 \bld{\alpha}(C_{A,B}) = \max(2a, 2b),
 $$
 whereas
 $$
 \bld{\alpha}(N_{A,B})=\max(a^2+\sigma^2, b^2)
 \ {\rm and} \
 \bld{\alpha}(M_{A,B}) = \max(2a+\sigma^2, 2b)
 $$
 can be arbitrarily large.

  An important question in the study of
 stochastic systems is whether  the matrices $D_{A,B}$ and
 $C_{A,B}$ are stable, that is,
 $
 \bld{\rho}(D_{A,B})<1$ and
 $ \bld{\alpha}(C_{A,B})<0.
 $
 One consequence of Propositions \ref{Prop-dtc} and \ref{prop-CT} is that
 stability of the stochastic Kronecker sum matrix is equivalent to the
  mean-square asymptotic stability of the
 corresponding stochastic system:
 \begin{align*}
 &\bld{\rho}(D_{A,B})<1\
 \Leftrightarrow \ \lim_{n\to \infty} \bE |\bld{x}(n)|^2=0,\\
 &\bld{\alpha}(C_{A,B})<1 \
 \Leftrightarrow\ \lim_{t\to +\infty} \bE |\bld{x}(t)|^2=0.
 \end{align*}

 The example shows that it is possible to have this stability
 even when the matrices $N_{A,B}$ and $M_{A,B}$
 are not stable: $D_{A,B}$ is stable if (and only if)  $\max(|a|,|b|)<1$,
 and $C_{A,B}$ is stable if (and only if) $\max(a,b)<0$; this is also
 clear by looking directly at the corresponding stochastic system.

  One can always use the
 lower bounds in \eqref{DTSKS} and \eqref{CTSKS} to
 check if the  the matrices $D_{A,B}$ and
 $C_{A,B}$ (and hence the corresponding systems)
 are not stable. In the above example,
 if
 $$
 \bld{\varrho}(N_{A,B})=\min(a^2+\sigma^2,b^2)>1,
 $$
  then $|b|>1$ and  $D_{A,B}$ is certainly not stable;
similarly,
 if
 $$
 \bld{\varrho}(M_{A,B})=\min(2a+\sigma^2,2b)>0,
 $$
 then $b>0$ and $C_{A,B}$ is certainly  not stable.


\def\cprime{$'$}
\providecommand{\bysame}{\leavevmode\hbox to3em{\hrulefill}\thinspace}
\providecommand{\MR}{\relax\ifhmode\unskip\space\fi MR }
\providecommand{\MRhref}[2]{%
  \href{http://www.ams.org/mathscinet-getitem?mr=#1}{#2}
}
\providecommand{\href}[2]{#2}

\vskip 0.2in

\end{document}